\documentclass[11pt,letterpaper]{amsart}
\pdfoutput=1

\usepackage{amsopn,amssymb}
\usepackage{lpic}
\usepackage{caption}
\usepackage{hyperref}
\hypersetup{pdfauthor={Jonah Gaster},pdftitle={lifting curves simply},colorlinks=true,linkcolor=black,citecolor=black}

\renewcommand{\H}{\mathbb{H}}

\numberwithin{equation}{section}

\theoremstyle{plain}
\newtheorem{thm}{Theorem}
\newtheorem{cor}[thm]{Corollary}

\newtheorem{lemma}[thm]{Lemma}
\newtheorem{prop}[thm]{Proposition}

\newenvironment{remark}[1][Remark.]{\begin{trivlist}
\item[\hskip \labelsep {\bfseries #1}]}{\end{trivlist}}

\theoremstyle{definition}

\theoremstyle{definition}

\begin{document}
\title{lifting curves simply}
\author{Jonah Gaster}

\address{
Mathematics Department\\
Boston College\\
\tt{gaster@bc.edu}
}

\date{January 1, 2015}

\begin{abstract}
{We provide linear lower bounds for $f_\rho(L)$, the smallest integer so that every curve on a fixed hyperbolic surface $(S,\rho)$ of length at most $L$ lifts to a simple curve on a cover of degree at most $f_\rho(L)$. This bound is independent of hyperbolic structure $\rho$, and improves on a recent bound of Gupta-Kapovich \cite{gupta-kapovich}. When $(S,\rho)$ is without punctures, using \cite{patel} we conclude asymptotically linear growth of $f_\rho$. When $(S,\rho)$ has a puncture, we obtain exponential lower bounds for $f_\rho$.} 
\end{abstract}

\maketitle

\section{introduction}

Let $S$ be a topological surface of finite type and negative Euler characteristic, and let $\rho$ be a complete hyperbolic metric on $S$. Let $\mathcal{C}(S)$ indicate the set of \emph{closed curves} on $S$, i.e.~the set of free homotopy classes of the image of immersions of $S^1$ into $S$. For $\gamma\in\mathcal{C}(S)$, let $\ell_\rho(\gamma)$ indicate the length of the $\rho$-geodesic representative of $\gamma$ on $S$, and let $\iota(\gamma,\gamma)$ indicate the geometric self-intersection number of $\gamma$. A closed curve $\gamma\in \mathcal{C}(S)$ is \emph{simple} when its self-intersection $\iota(\gamma,\gamma)$ is equal to zero.

It is a corollary of a celebrated theorem of Scott \cite{scott} that each closed curve $\gamma\in \mathcal{C}(S)$ lifts to a simple closed curve in some finite-sheeted cover (i.e.~$\gamma$ `lifts simply'). Recent work has focused on making Scott's result effective \cite{patel}. As such, for $\gamma\in \mathcal{C}(S)$, let $\deg(\gamma)$ indicate the minimum degree of a cover to which $\gamma$ lifts simply. 

We focus on two functions $f_\rho$ and $f_S$. Let the integer $f_S(n)$ be the minimum $d$ so that every curve $\gamma$ of self-intersection number $\iota(\gamma,\gamma)$ at most $n$ has degree $\deg(\gamma)$ at most $d$, and let the integer $f_\rho(L)$ be the minimum $d$ so that every curve $\gamma$ of $\rho$-length $\ell_\rho(\gamma)$ at most $L$ has degree $\deg(\gamma)$ at most $d$. Gupta-Kapovich have recently shown:
\begin{thm}\cite[Thm.~C, Cor.~1.1]{gupta-kapovich}
\label{gupta-kapovich}
{There are constants $C_1=C_1(\rho)$ and $C_2=C_2(S)$ so that $$f_{\rho}(L) \ge C_1 \cdot \left( \log L \right)^{1/3} \text{\ \ and \ \ }
f_S (n) \ge C_2 \cdot \left( \log n \right)^{1/3}.$$}
\end{thm}

Their work analyzed the `primitivity index' of a `random' word in the free group, exploiting the many free subgroups of $\pi_1S$ (e.g.~subgroups corresponding to incompressible three-holed spheres, or \emph{pairs of pants}) to obtain the above result. We also exploit the existence of free subgroups of $\pi_1S$, but instead of following in their delicate analysis of random walks in the free group, we analyze explicit curves on $S$. The chosen curves are sufficiently uncomplicated to allow a straightforward analysis of the degree of any cover to which the curves lift simply. As a consequence, we provide the improved lower bounds:

\begin{thm}
\label{main thm}
{We have $f_S(n) \ge n+1$. Moreover, Let $B=B(S)$ be a Bers constant for $(S,\rho)$, and let $\epsilon>0$. Then there is an $L_0=L_0(\rho,\epsilon)$ so that, for any $L \ge L_0$,
$$f_{\rho}(L) \ge \frac{L}{B+\epsilon}.$$}
\end{thm}
Recall the theorem of Bers \cite{bers}: There is a constant $B=B(S)$ so that, for every hyperbolic metric $\rho$ on $S$, there is a maximal collection of disjoint simple curves on $S$ with each curve of $\rho$-length at most $B$. Such a constant $B$ is called a \emph{Bers constant}, and such a collection of curves is called a \emph{Bers pants decomposition} for $(S,\rho)$. It is interesting to note that the constant in the lower bound for $f_\rho$ in Theorem \ref{main thm} is independent of the metric $\rho$.\vspace{.1cm}

\begin{remark}{The proof of Theorem \ref{main thm} follows from an analysis of an explicit sequence of curves $\{\gamma_n\}$. These curves are also analyzed by Basmajian \cite{basmajian}, where it is shown that they are in some sense the `shortest' curves of a given intersection number: The infimum of the length function $\ell(\gamma_n)$ on the Teichm\"{u}ller space of $S$ is asymptotically the minimum possible among curves with self-intersection $\iota(\gamma_n,\gamma_n)$ \cite[Cor.~1.4]{basmajian}.\vspace{.1cm}}
\end{remark}

Combined with work of Patel \cite[Thm.~1.1]{patel} (see the comment of \cite[p.~1]{gupta-kapovich}), in many cases Theorem \ref{main thm} implies a determination of the order of growth of $f_\rho$. We have:

\begin{cor}[Linear growth of $f_\rho$]
\label{main cor}
{Suppose $(S,\rho)$ is without punctures. There exist constants $C_1=C_1(S)$, $C_2=C_2(\rho)$, and $L_0=L_0(\rho)$ so that, for any $L \ge L_0$,
$$C_1\cdot L \le f_\rho (L) \le C_2\cdot L.$$}
\end{cor}

Recall that, when $S$ has boundary, we say that $(S,\rho)$ has a \emph{puncture} if the closed curve homotopic to a boundary component has no geodesic representative. When $(S,\rho)$ does have a puncture, we are not aware of upper bounds for $f_\rho$. In fact, the hypothesis above is essential. We show:

\begin{thm}
\label{punctured thm}
{Suppose $(S,\rho)$ is a hyperbolic surface with a puncture. For any $\epsilon>0$, there is $L_0=L_0(\rho,\epsilon)$ so that, for any $L\ge L_0$, $$f_\rho(L) \ge e^{\frac{L}{2+\epsilon}}.$$}
\end{thm}

This theorem indicates at least that Patel's upper bounds cannot hold in the punctured setting: If $(S,\rho)$ is a hyperbolic surface with a puncture, the minimal degree of a cover to which a given curve $\gamma$ lifts to a simple curve cannot be bounded linearly in the curve's length $\ell_\rho(\gamma)$.

There are other avenues for further investigation. It would be natural to seek upper bounds for $f_S(n)$ (cf.~\cite[p.~15]{rivin}), since no such bound follows from \cite{patel}. One might also investigate whether the constant $C_2$ in the upper bound in Corollary \ref{main cor} can be made independent of $\rho$, as with the lower bound. Finally, one could explore the set of curves of self-intersection number exactly $n$. For instance: Among the finitely many mapping class group orbits of curves $\gamma$ with self-intersection $n$, which maximize $\deg(\gamma)$? 

\subsection*{Outline of the paper} In \S\ref{degree section} we introduce a sequence of curves $\{\gamma_n\}$ on a pair of pants and analyze the degrees $\deg(\gamma_n)$, and in \S\ref{proofs section} we deduce Theorem \ref{main thm} and Theorem \ref{punctured thm} as straightforward consequences.

\subsection*{Acknowledgements}
The author gratefully acknowledges inspiration from Neha Gupta, inspiration and comments from Ilya Kapovich, and helpful comments from, and conversations with, Tarik Aougab, Ian Biringer, Martin Bridgeman, David Dumas, Peter Feller, Brice Loustau, and Priyam Patel. 

\section{analysis of a certain curve family}
\label{degree section}

Let $P_0$ be a pair of pants. Identify $\pi_1(P_0,p)$ with a rank-2 free group $F$, with generators $a$ and $b$ as pictured in Figure \ref{P0}. Let $\gamma_n$ indicate the closed curve given by the equivalence class of $a\cdot b^n$. 

\begin{figure}[h]
	\centering
	\Large
	\begin{lpic}{P0(9.5cm)}
		\lbl[]{127.5,53;$p$}
		\lbl[]{-7,52;$a$}
		\lbl[]{223,52.5;$b$}
	\end{lpic}
	\caption{The pair of pants $P_0$, with generators $a$ and $b$.}
	\label{P0}
\end{figure}

\begin{figure}[h]
	\centering
	\includegraphics[width=9.5cm]{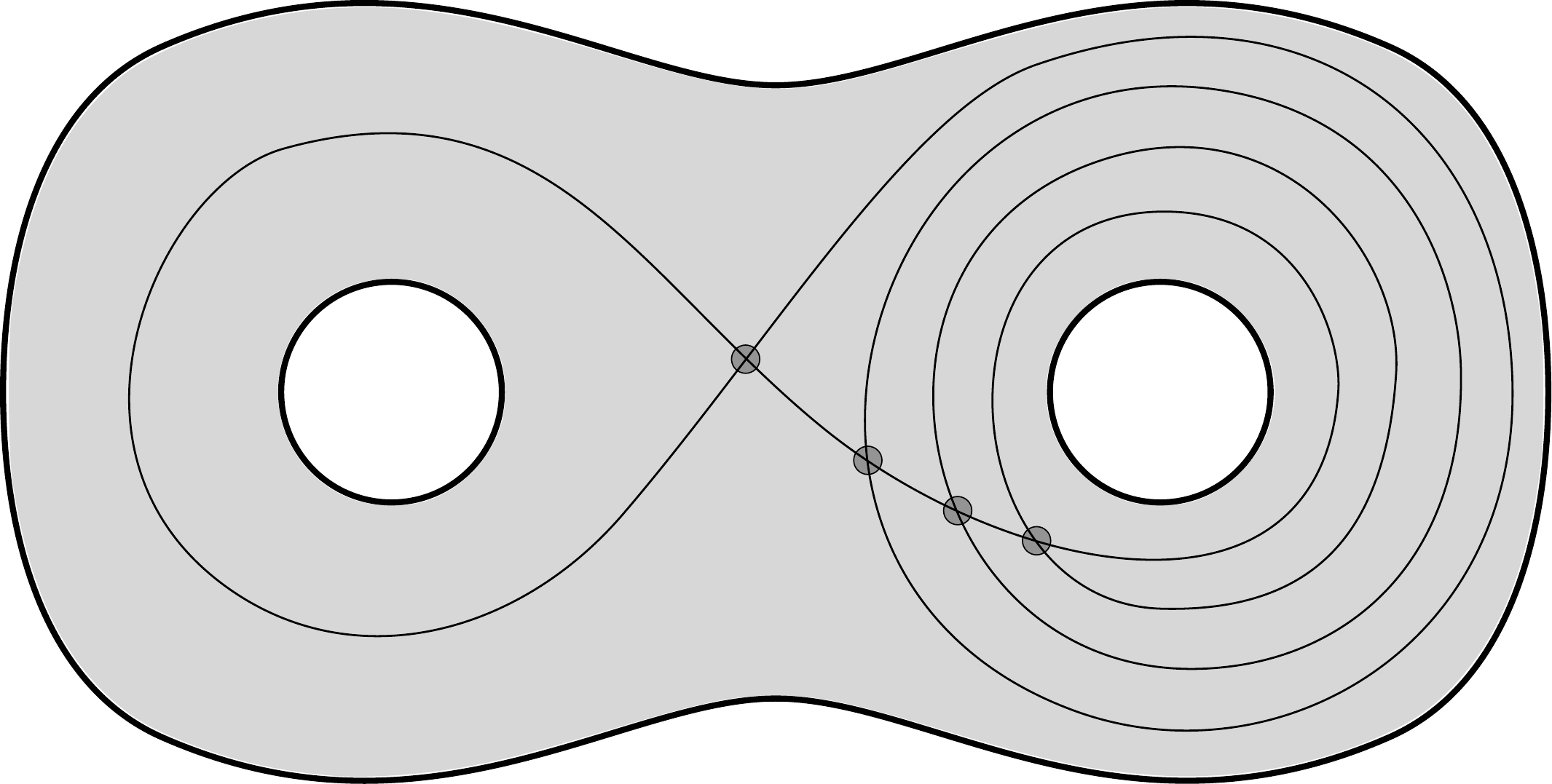}
	\caption{A minimal position representative of the curve $\gamma_4$.}
	\label{gamma_n}
\end{figure}

The following lemma is neither new (see \cite[Prop.~4.2]{basmajian}) nor surprising (see Figure \ref{gamma_n}), but we include a sketch of a proof for completeness: 

\begin{lemma}
\label{gamma non-simple}
{For $n\ge0$, the curve $\gamma_n$ has $\iota(\gamma_n,\gamma_n)=n$.}
\end{lemma}

\begin{proof}[Proof sketch]
{It is not hard to pick a representative of $\gamma_n$ that has self-intersection $n$, so that $\iota(\gamma_n,\gamma_n)\le n$ (see Figure \ref{gamma_n} for $n=4$). On the other hand, it is also not hard to check that there are no immersed bigons for this chosen representative of $\gamma_n$: For every pair of intersection points, the concatenation of any pair of arcs of $\gamma_n$ that connect the two points forms an essential curve. The `bigon criterion' of \cite[\S1.2.4]{farb-margalit} can be altered straightforwardly to an `immersed bigon criterion' in the setting of curves with self-intersections, and so the lack of immersed bigons guarantees that the chosen representative of $\gamma_n$ is in minimal position.}
\end{proof} 

\begin{figure}[h]
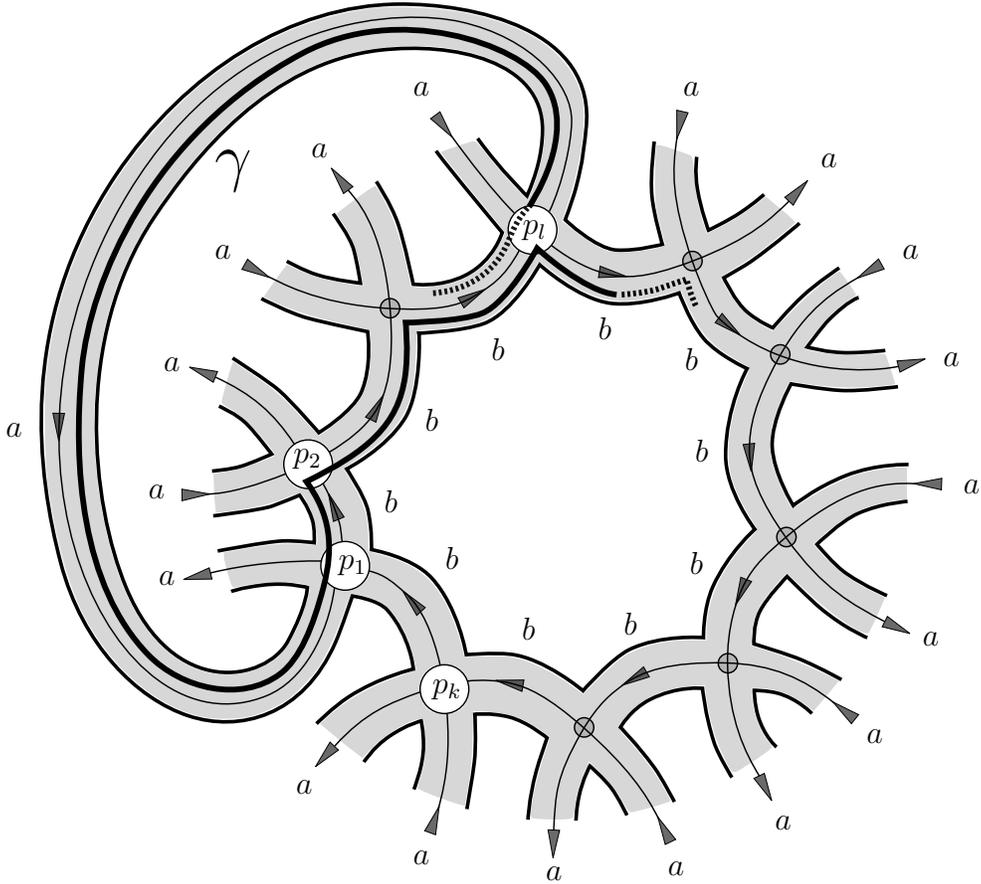

	\hspace{-.3cm}
	\centering
	\large
	\begin{lpic}{coverP3(12cm)}
		\lbl[]{61.5,57.5;$p_1$}
		\lbl[]{52.6,78.4;$p_2$}
		\lbl[]{97.2,123.3;$p_l$}
		\lbl[]{80,33;$p_k$}
		\lbl[]{-5,84;$a$}
		\lbl[]{81,59;$b$}
		\lbl[]{69,70;$b$}
		\lbl[]{77,86;$b$}
		\lbl[]{90,100;$b$}
		\lbl[]{111,104;$b$}
		\lbl[]{128,98;$b$}
		\lbl[]{130,80;$b$}
		\lbl[]{129,58;$b$}
		\lbl[]{116,46;$b$}
		\lbl[]{96,45;$b$}
		\lbl[]{25,55;$a$}
		\lbl[]{23,72;$a$}
		\lbl[]{26,97;$a$}
		\lbl[]{36,119;$a$}
		\lbl[]{55,139;$a$}
		\lbl[]{75,151;$a$}
		\lbl[]{128,151;$a$}
		\lbl[]{155,137;$a$}
		\lbl[]{171,119;$a$}
		\lbl[]{179,98;$a$}
		\lbl[]{183,73;$a$}
		\lbl[]{175,43;$a$}
		\lbl[]{164,24;$a$}
		\lbl[]{146,7;$a$}
		\lbl[]{125,-2;$a$}
		\lbl[]{101,-3;$a$}
		\lbl[]{75,0;$a$}
		\lbl[]{52,14;$a$}
		\Huge
		\lbl[]{38,135;$\gamma$}
	\end{lpic}		
	\vspace{.5cm}
	\caption{A supposedly simple lift $\gamma$ of $\gamma_n$ to the cover $P'\to P_0$, where $p_1$ is the point above $p$ that follows the unique $a$ edge of $\gamma$.}
	\label{coverP}
\end{figure}

We use Lemma \ref{gamma non-simple} to estimate $\deg(\gamma_n)$, a calculation reminiscent of \cite[Lemma 3.10]{gupta-kapovich}. The following proposition is the main tool in our analysis.

\begin{figure}
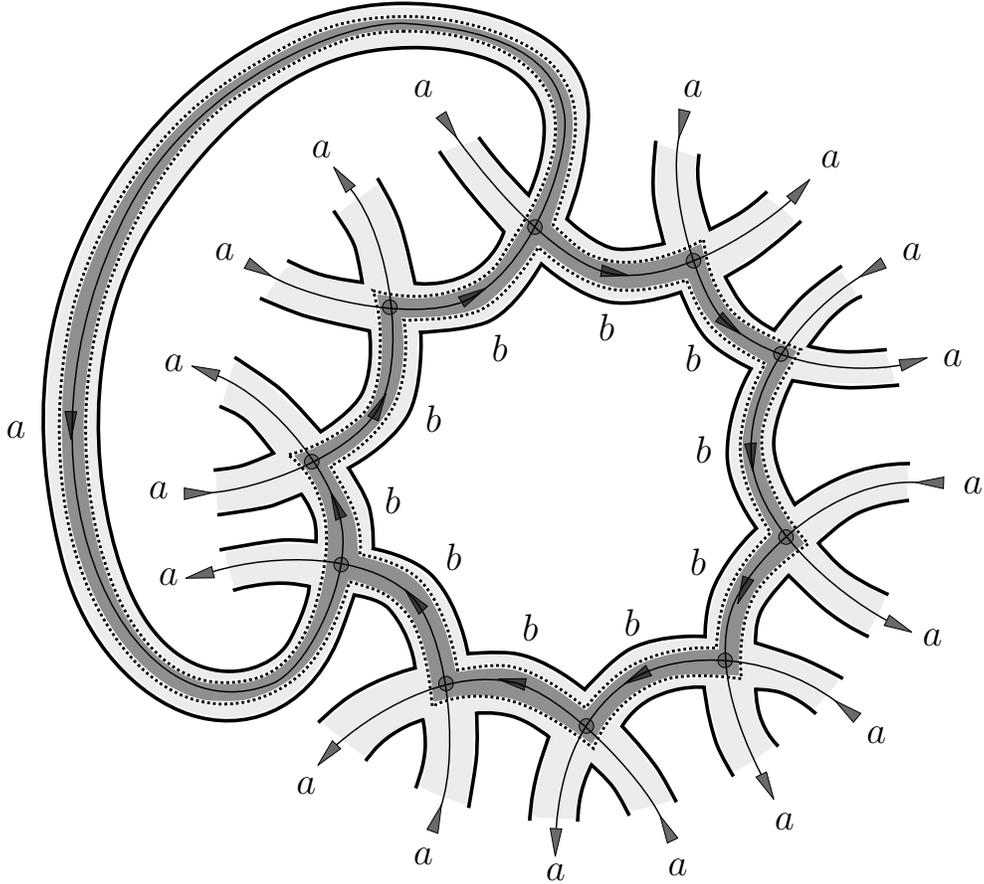

	\centering
	\Large
	\begin{lpic}{coverPwP0(12cm)}
		\lbl[]{-5,84;$a$}
		\lbl[]{81,59;$b$}
		\lbl[]{69,70;$b$}
		\lbl[]{77,86;$b$}
		\lbl[]{90,100;$b$}
		\lbl[]{111,104;$b$}
		\lbl[]{128,98;$b$}
		\lbl[]{130,80;$b$}
		\lbl[]{129,58;$b$}
		\lbl[]{116,46;$b$}
		\lbl[]{96,45;$b$}
		\lbl[]{25,55;$a$}
		\lbl[]{23,72;$a$}
		\lbl[]{26,97;$a$}
		\lbl[]{36,119;$a$}
		\lbl[]{55,139;$a$}
		\lbl[]{75,151;$a$}
		\lbl[]{128,151;$a$}
		\lbl[]{155,137;$a$}
		\lbl[]{171,119;$a$}
		\lbl[]{179,98;$a$}
		\lbl[]{183,73;$a$}
		\lbl[]{175,43;$a$}
		\lbl[]{164,24;$a$}
		\lbl[]{146,7;$a$}
		\lbl[]{125,-2;$a$}
		\lbl[]{101,-3;$a$}
		\lbl[]{75,0;$a$}
		\lbl[]{52,14;$a$}
	\end{lpic}
	\vspace{.2cm}
	\caption{An incompressible pair of pants $P''\subset P'$ contains $\gamma$.}
	\label{coverPwP0}
\end{figure}

\begin{prop}
\label{good curves}
{We have $\deg(\gamma_n)\ge n+1$.}
\end{prop}
\begin{proof}
{Towards contradiction, suppose there is a cover $P'\to P_0$ of degree $k \le n$, so that $\gamma_n$ lifts to a simple curve $\gamma$. Draw $P_0$ as a directed ribbon graph with one vertex $p$ and the two edges labeled by $a$ and $b$, and $P'$ as a directed ribbon graph with vertices $p_1,\ldots,p_k$ and $2k$ directed edges, $k$ with $a$ labels and $k$ with $b$ labels. Choose an orientation for $\gamma$ so that $\gamma$ consists of a directed $a$ edge followed by $n$ directed $b$ edges. After relabeling, we may assume that the unique $a$ edge of $\gamma$ is followed by $p_1$.

Starting from $p_1$ and reading the vertices visited by $\gamma$ in order, the vertex that immediately follows the $n$ consecutive $b$ edges of $\gamma$ is $p_l$, where $l$ is equivalent to $n+1$ modulo $k$. Finally, $\gamma$ follows an $a$ edge from $p_l$ to $p_1$. See Figure \ref{coverP} for a schematic.

This implies that there is an incompressible embedded pair of pants $P''$ in $P'$ that contains $\gamma$ (see Figure \ref{coverPwP0} in the case that $k \nmid n$ -- the other case is straightforwardly similar). After identifying $P''$ with $P_0$ appropriately, the closed curve $\gamma$ is given by the equivalence class of $a\cdot b^s$, where $s=\left\lfloor \frac{n}{k} \right\rfloor \ge 1$. By Lemma \ref{gamma non-simple} this curve is not simple, a contradiction.}
\end{proof}

\begin{remark}
{In fact, one can show that $\deg(\gamma_n) = n+1$, but the precise computation of $\deg(\gamma_n)$ is irrelevant.}
\end{remark}

\section{proofs of Theorem \ref{main thm} and Theorem \ref{punctured thm}}
\label{proofs section}

\begin{proof}[Proof of Theorem \ref{main thm}]
{First suppose that $P$ is any pair of pants on $S$, with any choice of identification of $P$ with $P_0$, so that we may view $\{\gamma_n\}$ as a sequence of closed curves on $S$. Suppose that $\pi:S'\to S$ is a cover of $S$ so that $\gamma_n$ lifts to a simple curve $\gamma'$. Let $P'$ be the component of $\pi^{-1}(P)$ containing $\gamma'$. We obtain a cover $\pi|_{P'}:P'\to P$, so that the degree of $S'\to S$ is at least the degree of $P'\to P$. By Proposition \ref{good curves}, the degree of $P'\to P$ is at least $n+1$. Thus $\deg(\gamma_n)\ge n+1$, and the bound for $f_S(n)$ follows immediately from Lemma \ref{gamma non-simple}.

We turn to the bound for $f_\rho(L)$. Let $P$ be a pair of pants with geodesic boundary in a Bers pants decomposition for the hyperbolic metric $\rho$. Let $\alpha$ and $\beta$ be two cuffs of $P$, and let $\delta$ indicate the simple arc connecting $\alpha$ to $\beta$. Let the $\rho$-length of $\delta$ be given by $\ell_\rho(\delta)=D$. Identify $P$ with $P_0$ so that $\alpha$ is in the conjugacy class of $a$ and $\beta$ is in the conjugacy class of $b$, and consider the closed curves $\{\gamma_n\}$ in $P$. Evidently, 
\begin{align*}
\ell_\rho(\gamma_n) & \le \ell_\rho(\alpha) + n\cdot \ell_\rho(\beta) + 2\ell_\rho(\delta) 
\le B(1+n) + 2D.
\end{align*}

Given $\epsilon>0$, for large $n$ the $\rho$-lengths satisfy $\ell_\rho(\gamma_n) \le n \cdot (B+ \epsilon).$ Let $$n=n(L)=\left \lfloor \frac{L}{B+\epsilon} \right\rfloor,$$ so that $\ell_\rho(\gamma_n) \le L$ for large enough $L$. Thus, for large enough $L$, we have $$f_\rho(L) \ge \deg(\gamma_n) \ge n+1 \ge \frac{L}{B+\epsilon}$$ as desired.}
\end{proof}

\begin{proof}[Proof of Theorem \ref{punctured thm}]
{Assume first that $(S,\rho)$ is not the three-punctured sphere. As before, we choose a Bers pants decomposition for $(S,\rho)$, letting $P$ be a pair of pants containing a puncture as a boundary component. Note that by assumption there is a pants curve of $P$ with hyperbolic holonomy. Identify $P$ with $P_0$ so that $b$ is homotopic to a curve that winds once around the puncture, and $a$ is homotopic to a pants curve with hyperbolic holonomy.  Consider again the sequence of curves $\{\gamma_n\}$ on $S$.

We assume the upper half plane model for the hyperbolic plane $\H^2$. By conjugating the holonomy representation of $\rho$ appropriately, we may arrange for the holonomy around the puncture to be the transformation $z\mapsto z+1$, and so that there is a lift of $a$ to the hyperbolic plane $\H^2$ that is contained in a Euclidean circle centered at $0$, say $|z|=s$. 

There is a lift of a curve freely homotopic to $b^n$ that starts at $is$, travels vertically along the imaginary axis to $iy$, travels horizontally to $n+iy$, and vertically down to $n+is$. Let $\beta'_y$ indicate the projection of this curve to $P$, and note that by construction its starting and ending point are in common, and on the geodesic cuff $\alpha$. We may thus concatenate (a parametrization of) $\alpha$ with $\beta'_y$, and the curve so obtained is homotopic to $\gamma_n$.

An elementary computation shows that $\ell_\rho(\beta'_y) = 2 \log (y/s) + n /y.$ Taking $y=n$ we find
$$\ell_\rho(\gamma_n)  \le \ell_\rho(\alpha) + \ell_\rho(\beta'_n) \le B - 2\log s + 1+ 2\log n.$$

Given $\epsilon>0$, for large $n$ the $\rho$-lengths satisfy $\ell_\rho(\gamma_n) \le (2+\epsilon) \log n$, and the result follows as in the proof of Theorem \ref{main thm}: Let $$n=n(L)= \left \lfloor e^{\frac{L}{2+\epsilon}} \right \rfloor,$$
so that $\ell_\rho(\gamma_n) \le L$ for large enough $L$. Thus, for large enough $L$, we have $$f_\rho(L) \ge \deg(\gamma_n) \ge n+1 \ge e^{\frac{L}{2+\epsilon}}$$ as desired.

If, on the other hand, $(S,\rho)$ is the unique hyperbolic structure on the three-punctured sphere, then we identify $P_0$ with $S$ arbitrarily. A straightforward calculation (see \cite[eq.~(29)]{basmajian}) shows that $$\ell_\rho(\gamma_n)=2\cosh^{-1}(1+2n).$$
The latter is asymptotic to $2\log n$ as $n$ goes to infinity, and so, given $\epsilon>0$, for large enough $n$ we have $\ell_\rho(\gamma_n) \le (2+\epsilon) \log n$. The result follows.
}
\end{proof}

\end{document}